\documentclass[a4paper, reqno]{amsart}

\usepackage[T1]{fontenc}
\usepackage{lmodern}
\usepackage{amssymb,amsmath,amsthm}
\usepackage[utf8]{inputenc}
\usepackage[unicode=true]{hyperref}
\hypersetup{
  breaklinks=true,
  pdfauthor={Anders Claesson and Stuart A. Hannah},
  pdftitle={Decomposing labeled interval orders},
  colorlinks=true,
  citecolor=blue,
  urlcolor=blue,
  linkcolor=violet,
 pdfborder={0 0 0}}
\usepackage{tikz}
\usetikzlibrary{matrix,arrows,patterns}

\newtheorem{theorem}{Theorem}

\newtheorem{proposition}[theorem]{Proposition}
\newtheorem{corollary}[theorem]{Corollary}

\theoremstyle{definition}
\newtheorem{definition}[theorem]{Definition}

\newtheorem{example}[theorem]{Example}

\newcommand{\sym}{\mathcal{S}}
\newcommand{\mati}[1]{\left[\!\begin{array}{c}#1\end{array}\!\right]}
\newcommand{\matii}[3]{\left[\!\begin{array}{cc}#1&#2\\ &#3\end{array}\!\right]}
\newcommand{\matiii}[6]{\left[\!\begin{array}{ccc}#1&#2&#3\\ &#4&#5\\ &&#6\end{array}\!\right]}
\newcommand{\eps}{\emptyset} 
\newcommand{\Bal}{\mathrm{Bal}}
\newcommand{\BalMat}{\mathrm{BalMat}}

\newcommand{\balpha}{\boldsymbol{\alpha}}
\newcommand{\bbeta}{\boldsymbol{\beta}}
\newcommand{\blambda}{\boldsymbol{\lambda}}
\newcommand{\bkappa}{\boldsymbol{\kappa}}
\newcommand{\beps}{\boldsymbol{\eps}}
\newcommand{\CC}{\mathcal{C}} 
\newcommand{\ITn}{\mathrm{InvTab}_n} 
\newcommand{\invtab}{\upsilon} 
\newcommand{\bkequiv}{\sim} 
\newcommand{\alequiv}{\approx} 
\DeclareMathOperator{\dent}{Dent} 
\newcommand{\iG}{H}
\newcommand{\irho}{\rho^{-1}}
\newcommand{\etal}{et al.\ }

\newcommand{\invTabOutline}[2]{
  \tikzstyle{disc} = [circle,thin,draw=black]
  \foreach \x in {#2}
  \draw (\x,#1-1-\x) -- (\x,#1-1) (#1-1-\x,\x) -- (#1-1,\x);
  \foreach \x in {#2}
  \draw (\x,#1) -- (\x,#1-1) (#1,\x) -- (#1-1,\x);
  \draw (0,#1) -- (0,#1-1) (#1,0) -- (#1-1,0);
  \draw(#1,0) -- (#1,#1) (0,#1) -- (#1,#1);
}

\newcommand{\pattern}[4]{
  \raisebox{0.6ex}{
  \begin{tikzpicture}
    [scale=0.35, baseline=(current bounding box.center), #1]
    \foreach \x/\y in {#4}
      \fill[pattern=north east lines, pattern color=black!45]
      (\x,\y) rectangle +(1,1);
    \draw (0.01,0.01) grid (#2+0.99,#2+0.99);
    \foreach \x/\y in {#3}
      \filldraw (\x,\y) circle (5pt);
  \end{tikzpicture}}\;
}

\newcommand{\nodestyle}{
  \tikzstyle{every node} = [font=\small];
}
\newcommand{\discstyle}{
  \tikzstyle{wht} =
    [ circle,fill=white,draw=black, minimum size=4.5pt, inner sep=0pt ];
}
\newcommand{\style}{
  \nodestyle
  \discstyle
}

\title[Decomposing labeled interval orders]{Decomposing labeled interval orders\\ as pairs of permutations}

\author{Anders Claesson}
\author{Stuart A. Hannah}
\address{Computer and Information Sciences\\ Livingstone
  Tower\\University of Strathclyde\\Glasgow\\Scotland}

\date{\today}


\begin{document}


\maketitle
\begin{abstract}
  We introduce ballot matrices, a signed combinatorial structure whose
  definition naturally follows from the generating function for labeled
  interval orders. A sign reversing involution on ballot matrices is
  defined. We show that matrices fixed under this involution are in
  bijection with labeled interval orders and that they decompose to a
  pair consisting of a permutation and an inversion table. To fully
  classify such pairs, results pertaining to the enumeration of
  permutations having a given set of ascent bottoms are given. This
  allows for a new formula for the number of labeled interval orders.
\end{abstract}



\section{Introduction} \label{introduction}
\thispagestyle{empty}

Recent work has employed the use of sign reversing involutions in the
study of unlabeled interval orders. Successes include taking
structures related to unlabeled interval orders directly to their
generating function \cite{levande-two-interp,yan} and
identifying statistical refinements \cite{remmel-tiefenbruck}.

In this paper we apply similar techniques to the labeled case.  We
introduce ballot matrices, a combinatorial structure consisting of
signed, upper triangular, non-row empty matrices whose entries are
ballots. The definition of such matrices follows naturally from the
generating function of labeled interval orders.  A bijection of Dukes
\etal\cite{comp-matrices-spec} is adapted to a surjection mapping
ballot matrices to labeled interval orders and used to define an
equivalence relation on ballot matrices.  A sign reversing involution
is then used to identify fixed points for which there is exactly one
per equivalence class. The decomposition of any single fixed point
into a pair consisting of a permutation and an inversion table is then
provided.  This allows for the main result of the paper, that the set
of labeled interval orders on $[n]$ is in bijection with two separate
sets. Firstly,
$$\{ (\pi,\tau) \in\sym_n \times \sym_n : A(\tau) \subseteq D(\pi) \},
$$
where $A(\tau)$ is the set of ascent bottoms of $\tau$, and $D(\pi)$ is
the set of descent positions of $\pi$. Secondly,
$$\{ (\pi, \tau) \in\sym_n \times \sym_n : D(\pi) \subseteq A(\tau) \}.
$$
As a consequence we derive a new formula for the number of
labeled interval orders on $[n]$:
\begin{align*}
  \sum_{\{s_1,\dots,s_k\}\subseteq [n-1]} \left(
    \det\left[ \binom{n-s_i}{s_{j+1}-s_i} \right]\cdot
    \prod_{r=1}^{k+1} r^{s_r - s_{r-1}}
  \right)
\end{align*}

where $s_0 = 0$ and $s_{k+1} =n$.

\subsection{Background}

A poset $P$ is said to be an \emph{interval order} if each $z \in P$
can be assigned a closed interval $[l_z,r_z] \in \mathbb{R}$ such that
$x<_P y$ if and only if $r_x < l_y$. Fishburn~\cite{fishburn}
demonstrates that interval orders are equivalently characterized as
posets with no induced subposet isomorphic to the pair of disjoint two
element chains, the so called $(2+2)$-free posets.

Bousquet-M{\'e}lou \etal\cite{two-plus-two-free} show that unlabeled
interval orders are in bijection with ascent sequences (a subset of
inversion tables), permutations avoiding the mesh pattern
$$\pattern{}{3}{1/2,2/3,3/1}{1/0,1/1,1/2,1/3,0/1,2/1,3/1},
$$
and a class of fixed point free involutions with no neighbor
nestings. Such involutions had previously been studied by
Zagier~\cite{zagier-chords} who determined their ordinary generating
function to be
$$\sum_{m \geq 0} \prod_{i=1}^m (1 - (1-x)^i).
$$
Levande~\cite{levande-two-interp} and Yan~\cite{yan} independently
employ the use of sign-reversing involutions to provide direct
interpretations of structures related to unlabeled interval orders
from Zagier's function.

To study labeled interval orders, Claesson
\etal\cite{part-comp-matrices} introduce composition matrices. A
\emph{composition matrix} is an upper triangular matrix on some
underlying set $U$ whose entries are sets partitioning $U$ satisfying
that there are no rows or columns which contain only empty set
partitions. They show that composition matrices have exponential
generating function
$$\sum_{m \geq 0} \prod_{i=1}^m (1 - e^{-xi}),
$$
again a function originally considered by
Zagier~\cite{zagier-chords}. They present a one-to-one
correspondence between labeled interval orders and composition
matrices via the Cartesian product of ascent sequences and set
partitions.

Dukes \etal\cite{comp-matrices-spec} give a direct bijection between
composition matrices and interval orders, where the downsets of
elements within the interval order are determined by hooks occurring
below the diagonal of the matrix.


\section{Terminology and preliminaries}\label{terminology}

Throughout this text, for non-negative integers $a$ and $b$ with $a<b$,
let $[b]$ denote the set $\{ 1,\dots,b \}$ and $[a,b]$ the set
$\{a,\dots,b\}$.  This paper will feature three main combinatorial
structures: permutations, inversion tables and ballots. In this section
a summary is provided to remind the reader of relevant results
pertaining to these structures and to set the notational convention that
shall be followed.

\subsection{Permutations}
A permutation is a bijection on a finite set. A descent in a
permutation $\pi = a_1 a_2 \dots a_n \in \sym_n$ is a pair
$(a_i,a_{i+1})$ where $a_i > a_{i+1}$.
Following Stanley~\cite[Section 2.2]{enumerative-combinatorics-one}
let $D(\pi) = \{ i : a_i <a_{i+1}\}\subseteq [n-1]$ denote the set of
descent positions and define
\begin{alignat*}{3}
  \balpha_n(S) &= \{ \pi \in \sym_n:  D(\pi) \subseteq S \},\quad
  &\alpha_n(S)&=|\balpha_n(S)|, \\
   \bbeta_n(S) &= \{ \pi \in \sym_n:  D(\pi) = S \},\quad
  &\beta_n(S)&=|\bbeta_n(S)|.
\end{alignat*}
Let $S=\{s_1,s_2,\dots,s_k\}$ and $1\leq s_1<s_2<\dotsb<s_k<n$.  Also,
let $s_0 = 0$ and $s_{k+1}=n$. Partitioning $[n]$ into blocks of
cardinalities
$$s_1-s_0,\, s_2-s_1,\, \dots,\, s_{k+1}-s_k
$$
a permutation is formed by listing elements
within the blocks in increasing order and concatenating the blocks. The
only position in which a descent \emph{can} occur is at the join
between two blocks. Thus,
\begin{equation}\label{equation-alpha}
\alpha_n(S) = \binom{n}{s_1-s_0, s_2-s_1, \dots, s_{k+1}-s_k}.
\end{equation}
By the sieve principle we have that $\beta_n(S) = \sum_{T\subseteq S}
(-1)^{|S\setminus T|} \alpha_n(T)$.  One can show
\cite[Example 2.2.4]{enumerative-combinatorics-one} that this leads to the
formula
$$
\beta_n(S)
= \det\left[ \binom{n-s_i}{s_{j+1}-s_i} \right],
$$
where $(i,j)\in [0,k]\times [0,k]$.

\subsection{Inversion tables}

Given a permutation $\pi = a_1 a_2 \dots a_n$, an inversion in $\pi$
is a pair $(a_i,a_j)$ where $a_i > a_j$ and $i<j$.
An inversion table is an encoding of a permutation where the $i$th
value is the number of inversions in which $i$ is involved as the
smaller element. The set of inversion tables of length $n$ will be
denoted $\ITn$:
$$\ITn = \{\, b_1 b_2 \dots b_n  : b_i \in [0,n-i] \,\}.
$$
An inversion table may be viewed diagrammatically. To make clear the
relationship between inversion tables and $n$ by $n$ upper triangular
matrices containing exactly one entry per row we shall break
convention and view an inversion table as right aligned, decreasing
rows where an entry in row $i$ at column $j$ corresponds to the
inversion table with $i$th entry $n-j$. An example is shown in
Figure~\ref{figure-inversion-table}.

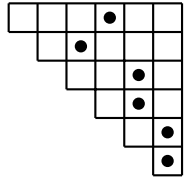
\begin{figure}
  \begin{center}
    \begin{tikzpicture}[line width=0.8pt, scale=0.38]
      \invTabOutline{6}{1,2,3,4,5}
      \foreach \x [count=\i]\y in {4,3,5,5,6,6}
      \draw[fill=black] (\x -0.5,5.5-\i+1) circle (5pt);
    \end{tikzpicture}
    \caption{Inversion table $231100$}
    \label{figure-inversion-table}
  \end{center}
\end{figure}

Define $\dent$ to be the function taking an inversion table to the set of
distinct entries it contains. For example, $\dent(430200) = \{0,2,3,4\}$.
We further say that $a\in [n-1]$ is \emph{missing} from a length $n$
inversion table if $a$ is not in its set of distinct entries.
For instance, $1$ and $5$ are both missing from $430200$.

\subsection{Ballots}

A \emph{ballot}, alternatively known as an ordered set partition, is a
collection of pairwise disjoint non-empty sets (referred to as blocks)
where the blocks are assigned some total ordering.  Adopting a
symbolic (or species) approach, let $L$ be the construction taking a
set $U$ to the set of linear orders built upon $U$. Also, let $E_+$ be
the non-empty set construction. That is, $E_+[U] = \{U\}$ if $U$ is
non-empty, and $E_+[\eps]=\eps$. Then define $\Bal$, the construction
of ballots, to be the composition $L(E_+)$:
$$\Bal = L(E_+) =  \sum_{k\geq 0} (E_+)^k.$$ 

Consider \emph{signed ballots}, as above but where each ballot is
assigned to be either positive or negative.  A positive ballot
contains an even number of blocks and a negative ballot contains an
odd number of blocks. For any species $F$, let $-1\cdot F =-F$ be as
$F$ but with the sign of each object negated. Using $E^{-1}$ to refer
to signed ballots---the notation stemming from its role as the
symbolic multiplicative inverse of set---we have
$$E^{-1} = L(-E_+) = \sum_{k\geq 0} (-1)^k(E_+)^k.
$$
It follows that signed ballots have exponential generating function
\begin{equation}\label{Einv}
   \frac{1}{1 + (e^x-1)} = e^{-x} = \sum_{n \geq 0} (-1)^n \frac{x^n}{n!}.
\end{equation}
See, for example, Bergeron \etal\cite[Section 2.5]{bergeron}.

We use the notation $(E^{-1})^+$ to refer to the subset of signed
ballots which are positive and $(E^{-1})^-$ to refer to the subset
which are negative.
 

\section{Ballot matrices and interval orders} \label{interval-orders-and-ballot-matrices}

Equation~\eqref{Einv} implies that the number of ballots constructed
on some set $U$ with an even number of blocks differ from the number
of ballots of $U$ with an odd number of blocks by $1$. To be precise
$$|(E^{-1})^{+}[U]| - |(E^{-1})^{-}[U]| = (-1)^{|U|}.
$$
An involution on ballots witnesses this fact. In the above equation
the sign of a ballot with $k$ blocks is $(-1)^k$.  Note that we can
change the sign of a ballot with $|U| \geq 2$ by splitting a
non-singleton block into two blocks or by merging two blocks. Let
$\omega=B_1\dots B_k$ be a ballot in $\Bal[U]$. That is, each $B_i$ is
non-empty and $U$ is the disjoint union of the sets $B_1$ through
$B_k$.

Take any linear order on $U$. Let $x = \min U$ be smallest element of
$U$. If $x\in B_i$ and $B_i$ contains at least two elements, then
delete $x$ from $B_i$ and create a new block $\{x\}$ to the immediate
right of $B_i$. For example,
$$\omega =
\{2,5\}\{1,4,6\}\{3\}\;\mapsto\; \{2,5\}\{4,6\}\{1\}\{3\}
= \xi.
$$
If $B_i=\{x\}$ and $i>1$ then delete this block from $\omega$ and add
$x$ to $B_{i-1}$. With $\omega$ and $\xi$ as in the example above, we
have $\xi\mapsto\omega$. If $B_1=\{x\}$ then proceed with the next
smallest element of $U$ and the ballot $B_2 B_3\dots B_k$. For example,
$$\{1\}\{2\}\{5\}\{4,6\}\{3\} \mapsto \{1\}\{2\}\{5\}\{3,4,6\}.
$$
For $U=\{u_1,u_2,\dots,u_n\}$ and $u_1< u_2 <\dots < u_n$ the
single fixed point under this sign reversing involution is
$\{u_1\}\{u_2\}\dots\{u_n\}$.

\subsection{Ballot Matrices}

The exponential generating function for the number of labeled interval
orders was shown by Claesson \etal\cite{comp-matrices-spec} to be a
function originally studied by Zagier~\cite{zagier-chords},
$$ \sum_{m\geq 0} \prod_{i=1}^m (1 - e^{-xi}) = \sum_{m\geq 0} (-1)^m\prod_{i=1}^m (e^{-xi} - 1).
$$

It it thus natural to consider the signed combinatorial structure
$$ \sum_{m\geq 0} (-1)^m\prod_{i=1}^m \big((E^{-1})^i - 1\big).
$$ An $((E^{-1})^i - 1)$-structure is a non-empty sequence of $i$
pairwise disjoint ballots. As such, a $(-1)^m\prod_{i=1}^m
\big((E^{-1})^i - 1\big)$-structure is an upper triangular $m\times m$
matrix of pairwise disjoint ballots such that each row is non-empty.

The sign of the matrix is the product of the signs of the
ballot entries and the signs of the rows. If $A$ is such a matrix and
the total number of blocks of all ballots in $A$ is $\ell$, then the sign
of $A$ is $(-1)^{\ell+m}$.  We shall call such matrices \emph{Ballot
matrices} and use the notation $\BalMat$ for the construction with
$\BalMat^{+}$ and $\BalMat^{-}$ the positive and negative parts
respectively. As an example, for $U=\{1,2\}$ we have
$$
\arraycolsep=0.8ex
\BalMat^{+}[U] = \textstyle{\left\{
    \mati{\{1,2\}},
    \matii{\eps}{\{1\}}{\{2\}},
    \matii{\eps}{\{2\}}{\{1\}},
    \matii{\{2\}}{\eps}{\{1\}},
    \matii{\{1\}}{\eps}{\{2\}}
  \right\}}
$$
and
$$\arraycolsep=0.8ex \BalMat^{-}[U]=\left\{ \mati{\{1\}\{2\}}, \mati{\{2\}\{1\}} \right\}.
$$

We note the similarity between ballot matrices and the composition
matrices of Claesson \etal\cite{part-comp-matrices}. The entries of
composition matrices are sets, which may be viewed as either
as ballots with a single block or as ballots where each element is contained
within its own singleton block and the blocks are ordered according to the
order on $U$.  Therefore composition matrices are a subset of ballot
matrices. For our purposes we wish to define an involution whose fixed
points are either all positive or all negative for any given
$U$. However for both interpretations of composition matrices as
ballot matrices the sign is not consistent, there exist both positive
and negative composition matrices when $|U| \geq 2$, and hence they
are not suitable candidates for the fixed points of our involution.

Dukes \etal\cite{comp-matrices-spec} provide a direct bijection
between composition matrices and labeled interval orders. We adapt
their mapping to define a surjection taking ballot matrices to labeled
interval orders as follows.

\begin{definition}\label{poset-definition}
  Let $A\in\BalMat[U]$, and let $x$ and $y$ be elements of
  $U$. Further, let $\omega$ and $\xi$ be the ballot entries $(i,j)$
  and $(i',j')$ of $A$ such that $x$ is contained in the underlying set
  of $\omega$ and $y$ is contained in the underlying set of
  $\xi$. Define the poset $P(A)$ by declaring that $x < y$ in $P$ if
  $j < i'$.
\end{definition}

In other words, $x<y$ in $P$ if the ``hook'' from $x$ to $y$ passing
through $(i',j)$ goes below the diagonal:
$$
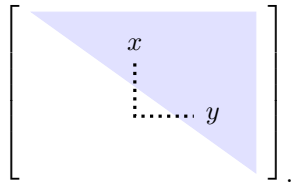
\begin{tikzpicture}
  \matrix [row sep=1.3ex,column sep=2.2ex, left delimiter={[},right delimiter={]}]
  {
    \node (b){}; & & &             & & & \node (a){}; \\
                 & & & \node(x){}; & & & \\
                 & & &             & & & \\
                 & & &             & & & \\
                 & & & \node(z){}; & &\node (y){}; & \\
                 & & &             & & & \\
                 & & &             & & & \node (c){}; \\
  };
  \filldraw [blue!12!white] (a.north east) -- (b.north west) -- (c.south east) -- cycle;
  \node at (x.north) {$x$};
  \node at (y.east) {$y$};
  \draw[dotted, very thick] (x.south) -- (z.center) -- (y.west);
\end{tikzpicture}.
$$

Equivalently, the strict downset of $y$ is the union of columns $1$
through $i'-1$. Figure~\ref{figure-poset} shows an example of a ballot matrix
and its corresponding poset.

\begin{figure}[h!]
  $$
  {\left[\!
  \begin{array}{cccc}
        \eps  & \{6\}   & \eps  & \{ 4,5 \} \\
              & \eps    & \{3\} & \eps      \\
              &         & \{1\} & \eps      \\
              &         &       & \{2\}     \\
  \end{array}\!
  \right]}\qquad\quad
  \begin{tikzpicture}[xscale=0.5, yscale=0.7, semithick, baseline=12]
    \style;
    \node [wht] (a) at (-1,0) {};
    \node [wht] (b) at ( 0,0) {};
    \node [wht] (c) at ( 1,0) {};
    \node [wht] (d) at ( 2,0) {};
    \node [wht] (e) at (-1,1) {};
    \node [wht] (f) at ( 0,2) {};
    \draw (a) node[below=2pt] {6} -- (e) node[left=2pt] {1} -- (f) node[right=2pt] {2};
    \draw (b) node[below=2pt] {3} -- (f);
    \draw (c) node[below=2pt] {4};
    \draw (d) node[below=2pt] {5};
  \end{tikzpicture}
  $$
  \caption{A ballot matrix and its corresponding poset}
  \label{figure-poset}
\end{figure}
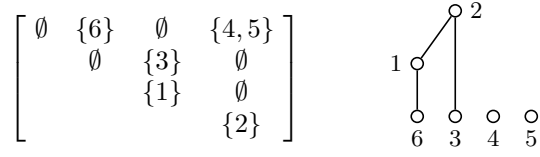

Given a poset $P$, the downset of $x \in P$ is the set of elements
smaller than $x$:
$$D(x) = \{ y \in P : y<x \}.
$$
It is a well known that a poset is an interval order if and
only if there is a linear ordering by inclusion on the downsets of
each element $\{D(x): x \in P \}$ (see, for example,
Bogart~\cite{bogart}). As the mapping states that the strict downset
of $y$ is the union of columns $1$ through $i'-1$ there is a linear
ordering on downsets and hence every poset which is mapped to must be
an interval order.

In addition, composition matrices are a subset of ballot matrices and
as Dukes \etal\cite{comp-matrices-spec} show that for composition
matrices the mapping is a bijection it follows that the adapted
mapping is a surjection.

If we declare that two ballot matrices in $\BalMat[U]$ are equivalent if
they determine the same interval order, then, by definition, there are
as many equivalence classes as there are interval orders on $U$. In the
next section we define as sign reversing involution that respects this
equivalence relation.


\section{The involution}\label{the-involution}

We now define the involution on ballot matrices. We begin by applying
the ballot involution componentwise to entries of $\BalMat$.

Choose a canonical linear order of the entries of the matrix; for
instance, order the entries (ballots) with respect to their minimum
element, or order them lexicographically with respect to their
position $(i,j)$ in the matrix. Then apply the ballot involution to
the first entry that is not fixed, if such an element exists, and
denote this operation $\psi$. A matrix is a fixed point under this
sign reversing involution if and
only if each entry of the matrix is fixed, and thus of the form
$$\{a_1\}\{a_2\}\dots\{a_j\}
\quad \mbox{with} \quad a_1 < a_2 < \dots< a_j.
$$
Note that if $A$ is a $k\times k$ matrix fixed under $\psi$, then the
sign of $A$ is $(-1)^{n+k}$, where $n=|U|$. We shall define a sign
reversing involution $\varphi$ on the fixed points of $\psi$.

Let $A\in\BalMat[U]$ be a matrix fixed under $\psi$.  Let $x\in U$ and
assume that that $x$ is on row $i$ and column $j$ of $A$. We say that
$x$ is a \emph{pivot} element of $A$ if row $i$ contains at least two
elements of $U$ and $x$ is the smallest element on row $i$, or the
following three conditions are met:

\begin{enumerate}
\item
  column $i$ is empty;
\item
  $\{x\}$ is the only non-empty ballot on its row;
\item
  $x$ is smaller than the minimum element of row $i+1$ of $A$.
\end{enumerate}

As an illustration, the pivot elements of the matrix
$$\left[\!
\begin{array}{ccccc}
  \eps & \{4\}      & \eps & \eps       &\eps  \\
       & \{6\}\{8\} & \eps & \{3\}\{7\} &\eps  \\
       &            & \eps & \{2\}      &\eps  \\
       &            &      & \{9\}      &\{5\} \\
       &            &      &            &\{1\}
\end{array}
\!\right]
$$
are $2$, $3$ and $5$.

If the set of pivot elements of $A$ is empty, then let
$\varphi(A)=A$. Otherwise, let $x$ be the smallest pivot element of $A$,
and assume that $x$ belongs to the $(i,j)$ entry of $A$.

\begin{enumerate}
\item
  If there is more than one element on row $i$, then remove $x$ from row
  $i$ and make a new row immediately above row $i$ with the block
  $\{x\}$ in column $j$ and the rest of the entries empty. Also insert a
  new empty column $i$, pushing the existing columns one step to the
  right.\medskip
\item
  If column $i$ is empty, $\{x\}$ is the only non-empty ballot on its
  row, and $x$ is smaller than the minimum element of row $i+1$, then
  remove column $i$ and merge row $i$ with row $i+1$ by inserting the
  singleton block ${x}$ at the front of the ballot in position $(i+1,j)$.
\end{enumerate}

Applying $\varphi$ to the example matrix above we get
$$\left[\!
\begin{array}{ccccc}
  \eps & \{4\}      & \eps       &\eps  \\
       & \{6\}\{8\} & \{3\}\{7\} &\eps  \\
       &            & \{2\}\{9\} &\{5\} \\
       &            &            &\{1\}
\end{array}
\!\right].
$$
Note that the smallest pivot element of this matrix is still $2$, and
applying $\varphi$ to it would bring back the original matrix.

Our main involution $\eta:\BalMat[U]\to\BalMat[U]$ is then defined as
the composition of $\psi$ and $\varphi$ in the following sense:
$$\eta(A) =
\begin{cases}
  \varphi(A) &\text{ if $\psi(A) = A$},\\
  \psi(A)    &\text{ if $\psi(A) \neq A$.}
\end{cases}
$$
It is clear that $\eta$ is sign reversing. That any fixed point of
$\eta$ has positive sign will be seen in
Section~\ref{the-fixed-points}.

\begin{proposition}\label{involution-preserves-interval-order}
  The involution $\eta$ preserves the interval order in the following
  sense. Let $A \in \BalMat[U]$. Let $P$ and $Q$ be the interval orders
  corresponding to $A$ and $\eta(A)$, respectively. Then $P = Q$.
\end{proposition}

\begin{proof}
  If $A$ is a fixed point of $\eta$, equality is immediate. Further, the
  block structure of the elements of $A$ is immaterial to the definition
  of the poset. Thus, if $\psi(A) \neq A$ and $\eta(A) = \psi(A)$, then
  equality is immediate. For the remainder of the proof assume that
  $\eta(A)=\varphi(A)\neq A$.

  The proof that the involution preserves the interval order is
  equivalent to saying that the strict downset of each element is
  preserved. This follows from a case analysis. Recall that the strict
  downset of $x$ at position $(i,j)$ in the matrix is the union of
  columns $1$ through $i-1$.

  Let $B=\eta(A)$. The involution has two possibilities. If the minimal
  pivot element $x$ at position $(i,j)$ in $A$ is not the only element
  on its row, then $B$ is formed by initially inserting a new empty row
  above row $i$ and a new empty column before column $i$. The pivot
  element $x$ is moved to the new row maintaining its column and hence
  its strict downset is unchanged.

  We now demonstrate that the insertion of the new empty row at
  position $i$ and new empty column at position $i$ preserves hooks
  below the diagonal. For $y \not = x$ at position $(i',j')$
  in $A$ there are three possibilities.

  \begin{enumerate}
  \item The element $y$ is above the newly inserted row and to the left
    of the new column, i.e.\ $y$ remains at position $(i',j')$ in $B$
    with $i'<i$ and $j'<i$. Then the new column is inserted to the right
    of the columns which form the strict downset of $y$ and hence the
    downset is unchanged.\medskip

  \item The element $y$ is to the right of the newly inserted column and
    above the inserted row, i.e.\ $y$ is at position $(i',j'+1)$ in $B$
    with $i'<i<j'$. Again as $i'<i$ the new column is inserted to the
    right of the columns which form the strict downset of $y$ and the
    downset is unchanged.\medskip

  \item The element $y$ is below the newly inserted row and to the right
    but of the new column, i.e.\ $y$ is at position $(i'+1,j'+1)$ in $B$
    with $i<i'$ and $i<j'$. As $i<i'$, the number of columns which form
    the downset of $y$ is increased by 1.  The newly inserted column $i$
    is empty and therefore contributes no new entries. As $i<i'$ the
    previous rightmost column $i'-1$ is shifted one place to the right
    to column $i'+1$ in the new matrix. The downset of $y$ in $B$ is
    therefore the union of elements $1$ through $i'$ and hence the
    downset is unchanged.
  \end{enumerate}

  Note that $x$ remains the pivot element in the newly constructed
  matrix $B$, the only non-empty ballot on its row, and with column
  $i$ empty. Therefore showing that the second possibility of the
  involution preserves posets follows from taking the reverse of the
  above cases.

  As the strict downsets are equal the posets are equal.
\end{proof}


\section{Fixed points}\label{the-fixed-points}

A fixed point under the sign reversing involution $\eta$ on $\BalMat$ is
an $n \times n$ matrix with no pivot elements, equivalently a matrix
such that
\begin{enumerate}
\item
  there is exactly one element per row;
\item
  if $a<b$, with $a$ on row $i$, and $b$ on row $i+1$, then column $i$
  is non-empty.
\end{enumerate}
Note that the total number of blocks in such a matrix is $n$---each
element is in its own block---and thus it has sign $(-1)^{2n} = 1$,
positive.

Further, matrices which satisfy these conditions can be decomposed to a
pair consisting of a permutation and an inversion table:  As there is
exactly one element per row, a permutation $\pi = a_1 \dots a_n$ can be
read setting each $a_i$ the value held in row $i$.  As the matrix is
also upper triangular, the position of the element in a row specifies an
inversion table $b_1 b_2 \dots b_n$ where each $b_i$ is $n$ minus the
column in which the entry in row $i$ occurs.

As an example, consider the matrix below. It decomposes into the
permutation $4132$ together with the inversion table $2010$:
$$
\left[\!
  \begin{array}{cccc}
    \eps & \{4\} & \eps  & \eps  \\
    & \eps  & \eps  & \{1\} \\
    &       & \{3\} & \eps  \\
    &       &       & \{2\} \\
  \end{array}\!
\right]
\;\simeq\;
\left(\,4132,\;\;
  \begin{tikzpicture}[line width=0.8pt, scale=0.25, baseline=13pt ]
    \invTabOutline{4}{1,2,3}
    \foreach \x [count=\i] in {2,0,1,0}{
      \draw[fill=black] (3.5 - \x,4.5 - \i) circle (5pt);}
  \end{tikzpicture}\,\,\right)
\;\simeq\;
\left(\,4132,\,2010\,\right).
$$

Take the equivalence class  on ballot matrices where two matrices are
equivalent if they correspond to the same interval order. We wish to
show that there is exactly one fixed point under $\eta$ per equivalence
class. For this purpose and to make explicit the link to previous
work we provide a bijection between composition matrices and
ballot matrices.

For the following, take the structure of the entries of a composition
matrix to be ballots where each element is contained within a
singleton block and the blocks are ordered according to the order on
the underlying set.

Given an $m \times m$ ballot matrix $A\in\BalMat[U]$, let $u_i$ be the
smallest element on the $i$th row of $A$, and define
$G(A) = U \setminus \{u_1,u_2,\dots,u_m\}$.

Assuming that $G(A)$ is non-empty, define $\rho$ to be the following
operation. Take $x = \min G(A)$ at position $(i,j)$ in $A$. Insert a
new row containing only empty ballots above row $i$ and a new column
containing only empty ballots to the left of column $i$. Move $x$ to
create a singleton ballot in the new row preserving its column. Note
that $|G(\rho(A))| = |G(A)|-1$. An example with $G(A)=\{4,6\}$ is given
below:
$$
\left[\!
  \begin{array}{cccc}
        \{2,6\} & \eps    & \eps & \eps    \\
                & \{3\}   & \eps & \boldsymbol{\{}\mathbf{4}\boldsymbol{\}} \\
                &         & \eps & \{1\}   \\
                &         &      & \{5\}   \\
  \end{array}\!
  \right]
\,\stackrel{\rho}{\longmapsto}\,
\left[\!
  \begin{array}{ccccc}
    \{2,6\} & \beps &\eps   & \eps  & \eps    \\
            & \beps &\beps  & \beps & \boldsymbol{\{}\mathbf{4}\boldsymbol{\}}\\
            &       & \{3\} & \eps  & \eps    \\
            &       &       & \eps  & \{1\}   \\
            &       &       &       & \{5\}   \\
  \end{array}\!
  \right].
$$

The inverse operation will be denoted $\irho$. To state it explicitly,
let $A\in\BalMat[U]$ be a $m\times m$ ballot matrix, and let $u_i$ be
the smallest element on the $i$th row of $A$, as before. Then take
$\iG(A)$ to be the subset of $\{u_1,u_2,\dots,u_{m-1}\}$ consisting of those
$u_i$ such that the following three conditions hold: column $i$ is
empty; $u_i$ is the sole element on row $i$; and $u_i>u_{i+1}$.

Assuming that $\iG(A)$ is non-empty, define $\irho$ to be the following
operation.  Take $x = \max \iG(A) $ at position $(i,j)$ in $A$. Append
$x$ in a singleton block at the end of the ballot in position $(i+1,j)$,
then remove row and column $i$.

\begin{proposition}\label{uniqueness-of-fixed-point}
  There is a bijection between composition matrices and ballot
  matrices fixed under $\eta$. As a result there is a unique ballot
  matrix fixed under $\eta$ per equivalence class.
\end{proposition}

\begin{proof}
  We first show that successive application of the mapping $\rho$
  gives an injection from composition matrices into ballot matrices
  fixed under $\eta$.

  The same argument as in
  Proposition~\ref{involution-preserves-interval-order} shows that
  $\rho$ preserves the interval order.

  Take a composition matrix. Let $A$ be the matrix returned after
  repeated application of $\rho$ until the set of elements $G(A)$
  is empty. We claim $A$ is a ballot matrix fixed under $\eta$.

  From definition we know that $G(A)$ is empty. Therefore there is
  exactly one element per row. The other requirement to be a fixed point
  under $\eta$ is that if $a<b$ with $a$ on row $i$ and $b$ on row $i+1$
  then column $i$ must be non-empty. As composition matrices have the
  property that all columns are non-empty and $\rho$ only introduces an
  empty column $i$ when $a>b$ with $a$ on row $i$, this requirement is
  met.

  Repeated application of $\rho$ is therefore a mapping between
  composition matrices and ballot matrices fixed under $\eta$ with
  injectivity following from the preservation of interval order.

  As $\rho$ preserves the interval order, the reverse operation
  $\irho$ also preserves the interval order.

  Take a fixed point matrix. Let $A$ be the matrix returned after
  repeated application of $\irho$ until the set of elements $\iG(A)$ is
  empty. We claim $A$ is a composition matrix.

  Composition matrices are neither row nor column empty. Non-row empty
  is a property of fixed point ballot matrices and $\irho$ does not
  introduce any empty columns. If a fixed point matrix contains an empty
  column $i$ then from definition there is an $a>b$ with $a$ and $b$ on
  rows $i$ and $i+1$ respectively. However as $G(A)$ is empty it follows
  that all empty columns are removed.

  Hence all fixed point matrices can be mapped to a composition
  matrices with the interval order preserved by repeated application
  of $\irho$, giving surjectivity.
\end{proof}

Let $\BalMat^\eta[U]$ denote the set of fixed points under $\eta$.
Writing simply $x$ for the ballot $\{x\}$, the complete list of matrices
in $\BalMat^\eta[3]$ is given in Figure~\ref{figure-list-of-fps}.
\begin{figure}
  \begin{align*}
  & \matiii{3}{\eps}{\eps}{2}{\eps}{1}
    \matiii{3}{\eps}{\eps}{\eps}{2}{1}
    \matiii{\eps}{3}{\eps}{2}{\eps}{1}
    \matiii{\eps}{3}{\eps}{\eps}{2}{1}
    \matiii{\eps}{\eps}{3}{2}{\eps}{1}\\
  & \matiii{\eps}{\eps}{3}{\eps}{2}{1}
    \matiii{3}{\eps}{\eps}{1}{\eps}{2}
    \matiii{\eps}{3}{\eps}{1}{\eps}{2}
    \matiii{\eps}{3}{\eps}{\eps}{1}{2}
    \matiii{\eps}{\eps}{3}{1}{\eps}{2}\\
  & \matiii{2}{\eps}{\eps}{3}{\eps}{1}
    \matiii{2}{\eps}{\eps}{\eps}{3}{1}
    \matiii{2}{\eps}{\eps}{1}{\eps}{3}
    \matiii{\eps}{2}{\eps}{1}{\eps}{3}
    \matiii{\eps}{2}{\eps}{\eps}{1}{3}\\
  & \matiii{\eps}{\eps}{2}{1}{\eps}{3}
    \matiii{1}{\eps}{\eps}{3}{\eps}{2}
    \matiii{1}{\eps}{\eps}{\eps}{3}{2}
    \matiii{1}{\eps}{\eps}{2}{\eps}{3}
  \end{align*}
  \caption{Complete list of matrices in $\BalMat^\eta[3]$}
  \label{figure-list-of-fps}
\end{figure}


\section{Permutations from ascent bottoms}\label{ascent-bottom}

In order to examine the fixed points under $\eta$ we shall consider
how to characterize the pairs resulting from their decomposition to a
permutation and an inversion table. For this purpose, this section is
concerned with counting the number of permutations whose set
of ascent bottoms is equal to some given set. Bijections between such
permutations and two different sets of inversion tables are
provided. We make repeated use of the sieve principle and our
presentation follows that of Stanley~\cite[Section
  2.2]{enumerative-combinatorics-one}.

Recall the definitions of $\balpha_n(S)$ and $\bbeta_n(S)$:
\begin{alignat*}{4}
  \balpha_n(S) &= \{ \tau \in \sym_n:  D(\tau) \subseteq S \},\quad
  &\alpha_n(S) &=|\balpha_n(S)|, \\
   \bbeta_n(S) &= \{ \tau \in \sym_n:  D(\tau) = S \},\quad
  &\beta_n(S)  &=|\bbeta_n(S)|.
\end{alignat*}
In an analogous fashion, for $\pi = a_1 a_2 \dots a_n \in \sym_n$, let
$$A(\pi) = \{ a_i : i \in [n-1], a_i < a_{i+1} \}
$$
be the set of ascent bottoms of $\pi$. Let
\begin{alignat*}{3}
   \bkappa_n(S) &= \{ \pi \in \sym_n:  A(\pi) \subseteq S \},\quad
  &\kappa_n(S)  &=|\bkappa_n(S)|, \\
  \blambda_n(S) &= \{ \pi \in \sym_n:  A(\pi) = S \},\quad
  &\lambda_n(S) &=|\blambda_n(S)|.
\end{alignat*}
Note that by definition  $\kappa_n(S) = \sum_{T\subseteq S} \lambda_n(T)$, and
by the sieve principle,
$\lambda_n(S) = \sum_{T\subseteq S} (-1)^{|S\setminus T|} \kappa_n(T)$.

The following set of sequences will be convenient as an intermediate
structure for later proofs.
\begin{definition}
  For fixed $n$, let $S= \{ s_1, \dots, s_k \}$ with $1\leq s_1< \dots <s_k<n$
  be given. Also, set $s_0=0$ and $s_{k+1}=n$. Define the Cartesian product
  $$\CC_n(S) = [0,k]^{s_{k+1}-s_k} \times \dots\times [0,1]^{s_2-s_1}\times [0,0]^{s_1-s_0}.
  $$
  We shall call an element of $\CC_n(S)$ a \emph{construction choice}.
\end{definition}

As example, for $n=8$ and $S=\{3,5,6,7\}$ we have $s_1-s_0=3$,
$s_2-s_1=2$, and $s_3-s_2=s_4-s_3=s_5-s_4=1$. Thus
$$\CC_n(S) =
[0,4]\times [0,3] \times [0,2] \times [0,1]
\times [0,1] \times [0,0] \times [0,0] \times [0,0].
$$ 
An example of a construction choice in $\CC_n(S)$ is $42001000$, we
shall use this as a running example throughout the remainder of this
section.

\begin{proposition}\label{kappa-perm}
  For fixed $n$, let $S = \{s_1, \dots, s_k\}$ and $1 \leq s_1< \dots <
  s_k < n$ be given. Then $\bkappa_n(S)$ is in bijection with $\CC_n(S)$.

\end{proposition}

\begin{proof}
  Take a construction choice $c_1 c_2 \dots c_n \in \CC_n(S)$.  We
  will use this to construct a permutation by insertion of entries at
  active sites.  Start with the empty permutation. This has a single
  active site, labeled zero. Reading the construction choice in
  reverse order, insert elements into the permutation beginning with
  the minimal element. That is, $c_i$ is the choice of active site for
  the insertion of $n+1-i$ into the permutation.

  A new active site is created when an element of $S$ is introduced into
  the permutation. The active sites are labeled according to the order
  in which they are inserted. That is, assuming entries of $S$ are
  numerically ordered then the active site to the right of $s_i$ in the
  permutation is labeled $i$. Note that a consequence of this is that
  $s_i$ is an ascent bottom if and only if $i$ is contained within the
  construction choice.  As a larger element is inserted at each step
  this ensures that the only place where an ascent can take place is
  after an entry of in the permutation which is contained within
  $S$. Therefore only elements of $S$ can be ascent bottoms.

  It is easy to see how to reverse this procedure and thus it provides
  the claimed bijection.
\end{proof}

\begin{example}\label{example-perm-from-construction-choice}

  For $n=8$ and $S=\{3,5,6,7\}$ the construction process for the
  permutation with construction choice $42001000$ is as follows. Note
  the new active site created when an element of $S$ is inserted.
  \begin{align*}
    &\,_0                  &\\
    &\,_01                 & \text{Insert 1 at site 0}\\
    &\,_021                & \text{Insert 2 at site 0}\\
    &\,_03_121             & \text{Insert 3 at site 0, contained in $S$}\\
    &\,_03_1421            & \text{Insert 4 at site 1}\\
    &\,_05_23_1421         & \text{Insert 5 at site 0, contained in $S$}\\
    &\,_06_35_23_1421      & \text{Insert 6 at site 0, contained in $S$}\\
    &\,_06_35_27_43_1421   & \text{Insert 7 at site 2, contained in $S$}\\
    &\,_06_35_27_483_1421 \quad \quad \quad & \text{Insert 8 at site 4}\\
  \end{align*}
  So the resulting permutation is $\pi = 65783421$, with $A(\pi)
  =\{3,5,7\}$.
\end{example}

\begin{corollary}\label{kappa-formula}
  For fixed $n$, let $S= \{ s_1, \dots, s_k \}$ with $1\leq s_1< \dots
  <s_k<n$ be given. Then
  $$\kappa_n(S) = \prod_{r=1}^{k+1} r^{s_r - s_{r-1}},
  $$
  where $s_0 = 0$ and $s_{k+1} = n$.
\end{corollary}
\begin{proof}
  By Proposition~\ref{kappa-perm} we have that $\kappa_n(S)$ is the
  cardinality of $\CC_n(S)$, from which the formula immediately follows.
\end{proof}

We shall now show that construction choices in $\CC_n(S)$, and thus
permutations in $\bkappa_n(S)$, are in bijection with two different
sets of inversion tables. Namely
$$\bigl\{ \invtab \in \ITn : \dent(\invtab) \subseteq \{0,s_1,s_2,\dots,
  s_k \} \bigr\}
$$
and
$$\bigl\{ \invtab \in \ITn : [n-1] \setminus \dent(\invtab) \subseteq
  \{n-s_1,\dots, n- s_k \} \bigr\}.
$$

\begin{proposition}\label{kappa-inversion-table-alternative}
  For fixed $n$, let $S= \{ s_1, \dots, s_k \}$ with $1\leq s_1< \dots
  <s_k<n$ be given. Then there is a bijection between $\bkappa_n(S)$ and
  inversion tables whose entries are a subset of $\{ 0 \} \cup S$,
  $$\bigl\{ \invtab \in \ITn : \dent(\invtab) \subseteq \{0,s_1,s_2,\dots,
  s_k \} \bigr\}.
  $$
\end{proposition}
\begin{proof}
  Again we shall use the construction choice.  Entries contained within
  the inversion table are a subset of $S$. Therefore elements which are
  in $[n-1]$ but \emph{not} in $S$, that is, elements of \ $[n-1]
  \setminus S$, cannot be contained in the inversion table. These
  entries are therefore forbidden. Label the remaining possible entries
  right to left from $[0,k]$. In this context it is convenient to use
  our diagrammatic representation of an inversion table. As an example,
  let $n=8$ and $S=\{ 3,5,6,7\}$. As $[n-1] \setminus S = \{ 1, 2, 4\}$,
  the columns $8-1$, $8-2$, and $8-4$ are forbidden (dark,
  below). Labeling those which remain right-to-left with $[0,4]$ yields
  
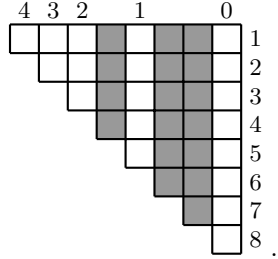
\vspace{-2ex}
  \begin{center}
    \begin{tikzpicture}[line width=0.8pt, scale=0.38]
      \tikzstyle{disc} = [circle,thin,draw=black]
        \foreach \x in {1,2,4}{
        \fill[black!40!white] (7-\x,8) rectangle (8-\x,\x);}
        \invTabOutline{8}{1,2,3,4,5,6,7}
        \foreach \x/\y in {0/0,3/1,5/2,6/3,7/4}{
          \node[] at (7.5-\x,8.5) {\small{\y}};}
        \foreach \x in {1,2,3,4,5,6,7,8}{
          \node[] at (8.5,8.5-\x) {\small{\x}};}
    \end{tikzpicture}.
  \end{center}
  \vspace{-1ex}
  Given a construction choice $c_1 c_2 \dots c_n\in \CC_n(S)$, assign
  the entry on row $i$ to be in the column labeled $c_i$. Note that as a
  consequence $s_i$ is contained in the inversion table if and only if
  $i$ is contained within the construction choice.  To consider the
  range of construction choices which are valid, we also note that there
  are $k+1$ allowed columns for the first $s_k-s_{k-1}$ rows, $k$
  choices for the next $s_{k-1}-s_{k-2}$ rows, and so on. This agrees
  with the definition of $\CC_n(n)$. Taking our example construction
  choice of $42001000$ yields the inversion table $\invtab = 75003000$
  where $\dent(\invtab) = \{0,3,5,7\}$:
  
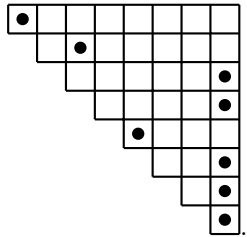
\vspace{-1.5ex}
  \begin{center}
    \begin{tikzpicture}[line width=0.8pt, scale=0.38]
      \tikzstyle{disc} = [circle,thin,draw=black]
        \invTabOutline{8}{1,2,3,4,5,6,7}
        \foreach \x [count=\i] in {7,5,0,0,3,0,0,0}{
        \draw[fill=black] (7.5 - \x,8.5 - \i) circle (5pt);}
    \end{tikzpicture}.
   \end{center}
  \vspace{-1ex}
\end{proof}

Applying the sieve principle to the set of inversion tables from
Proposition~\ref{kappa-inversion-table-alternative} we arrive at the
following result.

\begin{corollary}\label{lambda-inversion-table-alternative}
  There is a bijection between $\blambda_n(S)$ and inversion tables
  whose entries are exactly those in $\{0 \} \cup S$,
   $$\bigl\{ \invtab \in \ITn : \dent(\invtab) = \{0,s_1,\dots,
   s_k \} \bigr\}.
   $$
\end{corollary}

To prove the bijection between $\bkappa_n(S)$ and the second set of
inversion tables, consideration of a set of ballots is useful.  The
proof of Proposition~\ref{cc-ballot} below shows one way to make a
ballot in $\Bal[n]$ (short for $\Bal[[n]]$) from a given construction
choice.

\begin{proposition}\label{cc-ballot}
  For fixed $n$, let $S = \{s_1, \dots, s_k\}$ and $1 \leq s_1< \dots <
  s_k < n$ be given.  Then $\CC_n(S)$ is in bijection with the set of ballots
  $$\bigl\{\, B_1 \dots B_{k+1} \in \Bal[n]
  : \{\min B_1, \dots,\min B_{k+1} \} = \{1,s_1+1,\dots,
  s_k+1\} \,\bigr\}.
  $$
\end{proposition}

\begin{proof}
  We will show how to construct a ballot from a given construction
  choice $c_1c_2\dots c_n$. Take $k+1$ empty blocks. At any point in the
  following construction each block will be considered either open or
  closed, and the open blocks will be numbered $0$, $1$, \dots, $k$,
  from left to right. Initially all blocks are open. For $i$ equal to
  $1$, $2$, \dots, $n$, in that order, let $a=n+1-i$ and insert $a$ into
  the $c_i$th open block. If $a\in \{1,s_1+1,\dots, s_k+1\}$ then also
  close the block $a$ is inserted into. This way $a$ is guaranteed end
  up as the minimal element of its block. It is easy to see how to
  reverse this procedure and thus it provides the claimed bijection.
\end{proof}

\begin{example}\label{example-ballot}
  For $n=8$ and $S=\{3,5,6,7\}$ consider the construction of a ballot
  whose minimal block elements are $\{1,4,6,7,8 \}$ with construction
  choice $42001000$. Initially we have $5$ empty blocks labeled from
  $[0,4]$. Note that when a minimal block element is inserted, that
  block is no longer open and the remaining blocks are relabeled.
 \begin{alignat*}{3}
    &\{\}_0 \{\}_1 \{\}_2 \{\}_3 \{\}_4\\
    &\{\}_0 \{\}_1 \{\}_2 \{\}_3 \{8\}_4 \quad \quad &
    \text{$8$ inserted in block $4$, is minimal entry}\\
    &\{\}_0 \{\}_1 \{7\}_2 \{\}_3 \{8\} \quad \quad &
    \text{$7$ inserted in block $2$, is minimal entry}\\
    &\{6\}_0 \{\}_1 \{7\} \{\}_2 \{8\} \quad \quad &
    \text{$6$ inserted in block $0$, is minimal entry}\\
    &\{6\} \{5\}_0 \{7\} \{\}_1 \{8\} \quad \quad &
    \text{$5$ inserted in block $0$}\\
    &\{6\} \{5\}_0 \{7\} \{4\} \{8\} \quad \quad &
    \text{$4$ inserted in block $1$, is minimal entry}\\
    &\{6\} \{3,5\}_0 \{7\} \{4\} \{8\} \quad \quad &
    \text{$3$ inserted in block $0$}\\
    &\{6\} \{2,3,5\}_0 \{7\} \{4\} \{8\} \quad \quad &
    \text{$2$ inserted in block $0$}\\
    &\{6\} \{1,2,3,5\}_0 \{7\} \{4\} \{8\} \quad \quad &
    \text{$1$ inserted in block $0$, is minimal entry}
 \end{alignat*}
 Therefore the final ballot is $\{6\} \{1,2,3,5\} \{7\} \{4\} \{ 8\}$.
\end{example}

\begin{proposition}\label{kappa-inversion-table-original}
  There is a bijection between $\bkappa_n(S)$ and inversion tables
  whose missing elements are a subset of $n-s_1$,
  $n-s_2$, \dots, $n-s_k$,
  $$\bigl\{ \invtab \in \ITn : [n-1] \setminus \dent(\invtab) \subseteq
  \{n-s_1,\dots, n- s_k \} \bigr\}.
  $$
  Or, equivalently,
  $$\bigl\{ \invtab \in \ITn :[0,n-1] \setminus \{n-s_1, \dots,
  n-s_k\} \subseteq \dent(\invtab) \bigr\}.
  $$
\end{proposition}
\begin{proof}
  As seen in the proof of Equation~\eqref{equation-alpha} from
  Section~\ref{terminology}, a ballot can be taken to a
  permutation by writing the entries within a block in decreasing
  order and concatenating the blocks. By this method only the minimal
  element in a block may be an ascent bottom in the permutation, with
  the exception of the final block whose minimal element is the last
  element in the permutation.

  Hence, for a fixed $n$ and $S$, the ballot construction gives a
  bijection between permutations whose set of ascent bottoms is a
  subset of $S$ and permutations whose set of ascent bottoms plus the
  last element is a subset of $\{1\} \cup \{s_1 +1, \dots, s_k + 1 \}$.
  Let $\pi= a_1 \dots a_n$ be any such permutation. We shall denote
  the set of ascent bottoms plus the final element of $\pi$ as
  $T=\{t_1, t_2, \dots, t_j \}$:
  $$A(\pi) \cup \{  a_n\} = T \subseteq \{1\} \cup \{s_1 +1, \dots, s_k + 1 \}.
  $$
  An element in a permutation can either be an ascent bottom, a
  descent top, or the final element. Taking the complement of a
  permutation takes an ascent bottom $t_i$ to a descent top
  $n+1-t_i$. Letting $\pi^c$ denote the complement of $\pi$, it
  follows that for $\pi^c$ the set of descent tops and final element
  is
  $$\{n+1-t_1, n+1-t_2, \dots, n+1-t_j\} \subseteq \{n\} \cup \{
  n-s_1, \dots, n- s_k \},
  $$
  which contains at least the element $n$. The set of ascent bottoms in
  $\pi^c$ contains everything which is not a descent top or the final
  element.
  $$A(\pi^c) =[n] \setminus  \{n+1-t_1, \dots, n+1-t_j \}.
  $$
  As  $T \subseteq \{1\} \cup \{s_1 +1, \dots, s_k + 1 \}$, it
  follows that
  $$[n-1] \setminus \{n-s_1, \dots, n-s_k\} \subseteq A (\pi^c).
  $$
  From Corollary~\ref{lambda-inversion-table-alternative} we have that
  $\pi^c$ corresponds to an inversion table whose entries are
  exactly those in $\{0 \} \cup A(\pi^c)$, thus giving a unique
  inversion table satisfying
  $$[0,n-1] \setminus \{n-s_1, \dots, n-s_k\}
  \subseteq \dent(\invtab).
  $$
  This concludes the proof.
\end{proof}

\begin{example}
  As in previous examples, let $n=8$, $S=\{3,5,6,7 \}$ and consider the
  construction choice $42001000$. From Example~\ref{example-perm-from-construction-choice}
  the permutation in $\bkappa_n(S)$ that is
  given by the construction choice is $\pi = 65783421$. We wish to find
  the inversion table $\invtab$ corresponding to $\pi$ satisfying
  $$[0,7] \setminus \{ 8-3,8-5,8-6,8-7 \} = \{0,4,6,7\}
  \subseteq \dent(\invtab).
  $$
  From Example~\ref{example-ballot} the ballot given by the construction
  choice is $\{6\}\{1,2,3,5\}\{7\}\{4\}\{8\}$.  Writing the elements
  within a block in decreasing order and concatenating the blocks gives
  the permutation $\tau = 65321748$ with set of ascent bottoms $\{1,4\}$
  and final element $\{8\}$ where
  $$\{1,4,8\} \subset \{1, s_1+1, \dots, s_k+1 \} =
  \{1, 3+1, 5+1, 6+1, 7+1\}.
  $$
  The complement of $\tau$ is $\tau^c = 34678251 $ and has set of
  descent tops $\{9-4, 9-1\} = \{5,8\}$ and final element
  $9-8=1$. Every other entry in $\tau^c$ is an ascent bottom:
  $$A(\tau^c) = \{2,3,4,6,7\}.
  $$
  Taking $S'= A(\tau^c)$, it follows from
  Proposition~\ref{kappa-perm} that the construction choice
  uniquely specifying $\tau^c \in \bkappa_n(S')$ is
  $54312000$. Applying
  Proposition~\ref{kappa-inversion-table-alternative} and
  Corollary~\ref{lambda-inversion-table-alternative}, we can show
  that $\tau^c$ corresponds to the inversion table $76423000$,
  which, by construction, has set of distinct entries
  $$\dent(76423000) = \{ 0,2,3,4,6,7\} = \{0\} \cup A(\tau^c).
  $$
  Thus we have constructed $\invtab$ satisfying
  $\{0,4,6,7\} \subseteq  \{ 0,2,3,4,6,7\} = \dent(\invtab)$.
\end{example}




\section{Decomposition of fixed points}\label{decomp}

Recall that matrices fixed under the involution $\eta$
satisfy the properties

\begin{enumerate}
\item
  there is exactly one element per row;
\item
  if $a<b$, with $a$ on row $i$, and $b$ on row $i+1$, then column $i$
  is non-empty.
\end{enumerate}

Also recall that a fixed point matrix can be viewed as a pair
consisting of a permutation and an inversion table.

For $A\in \BalMat^\eta[U]$ where $n=|U|$, let $\pi(A) = a_1 \dots a_n$
be the permutation defined by setting $a_i$ the value held in the unique
nonzero element of row $i$ of $A$. Let an equivalence relation
$\bkequiv$ on $\BalMat^\eta[U]$ be defined by $A \bkequiv B$ if $\pi(A)
= \pi(B)$.

\begin{proposition}\label{bk-equiv-class-size}
  For $\pi \in \sym_n$, the equivalence class $[\pi]_\bkequiv$ is
  determined by the descent set $S = \{s_1, s_2, \dots, s_n \}=
  D(\pi)$ of $\pi$ alone.  In fact, fixed point matrices in
  $[\pi]_\bkequiv$ can be viewed as pairs consisting of the
  permutation $\pi$ and an inversion table whose set of missing
  entries is a subset of $\{n-s_1, n-s_2, \dots, n-s_k \}$.
\end{proposition}

\begin{proof}
  It is a defining property of a fixed point matrix that if $a<b$, with
  $a$ on row $i$, and $b$ on row $i+1$, then column $i$ is required to
  be non-empty. This is equivalent to saying that when the matrix is
  decomposed into a permutation and inversion table, that $n-i$ is an
  entry contained within the inversion table.

  So, if $a>b$ then we have a descent in the associated permutation and
  therefore column $i$ may or may not be empty. It follows that $n-i$
  may or may not be contained in the inversion table.

  Therefore, for $\pi\in\sym_n$, if the set of descent positions is
  $D(\pi) = S=\{s_1, s_2, \dots, s_k \}$, then the set of inversion
  tables with which $\pi$ can be paired are exactly those where the set
  of missing entries is a subset of $\{n-s_1, n-s_2, \dots, n-s_k\}$.
\end{proof}

\begin{theorem}\label{theorem-equiv-bk}
    Labeled interval orders on $[n]$ are in bijection with the set
    $$\sum_{S\subseteq [n-1]} \bbeta_n(S) \times \bkappa_n(S).
    $$
    This set may be alternatively written as
    $$ \{ (\pi, \tau) \in \sym_n \times \sym_n  : A(\tau) \subseteq
       D(\pi) \}.
    $$
\end{theorem}

\begin{proof}
  The adapted surjection of Dukes \etal is a bijection between labeled
  interval orders and fixed point ballot matrices. This is given by
  the equivalence class on ballot matrices according to interval order
  and Proposition~\ref{uniqueness-of-fixed-point} which shows that
  there is a unique fixed point per equivalence class.

  A fixed point matrix can be decomposed into a permutation $\pi$ and
  an inversion table. If $D(\pi) =\{s_1, s_2, \dots s_k\}$
  Proposition~\ref{bk-equiv-class-size} gives that the set of
  inversion tables with which $\pi$ can be paired are those whose set
  of missing elements is a subset of $\{n-s_1, n-s_2, \dots, n-s_k
  \}$.  We know from Proposition~\ref{kappa-inversion-table-original}
  that such inversion tables are in bijection with permutations in
  $\bkappa_n(D(\pi))$.
\end{proof}

\begin{corollary}\label{corollary-count-on-downset}
  The number of labeled interval orders on $[n]$ is
  given by the formula
  $$ \sum_{\{s_1,\dots,s_k\}\subseteq [n-1]} \left(
    \det\left[ \binom{n-s_i}{s_{j+1}-s_i} \right]\cdot
    \prod_{r=1}^{k+1} r^{s_r - s_{r-1}}
  \right),
  $$
  in which $s_0 = 0$ and $s_{k+1}=n$.
\end{corollary}

\begin{proof}
  This follows from the formula for $\beta_n$, see
  Stanley~\cite[Example 2.2.4]{enumerative-combinatorics-one}, and
  the formula for $\kappa_n$ given by Corollary~\ref{kappa-formula}.
\end{proof}

In the above we have taken the permutation to be fixed and considered
the set of inversion tables in the equivalence class under
$\bkequiv$. It is equally natural to instead take the inversion table as
fixed.

As before, for $A\in \BalMat^\eta[U]$, let $\invtab(A) = b_1 b_2 \dots
b_n$ be the inversion table from the decomposition of a ballot matrix
fixed under $\eta$ defined by setting $b_i$ to $n-j$ where $j$ is the
column of the only non-empty ballot entry on row $i$ of $A$.

Let the equivalence relation $\alequiv$ on $\BalMat^\eta[U]$ be
defined by $A \alequiv B$ if $\invtab(A) = \invtab(B)$.

\begin{proposition}\label{al-equiv-class-size}
  For $\invtab \in \ITn$, the equivalence class $[\invtab]_\alequiv$
  is determined by $\dent(\invtab)$ alone.  In fact, fixed point
  matrices in $[\invtab]_\alequiv$ can be viewed as pairs consisting
  of the inversion table $\invtab$ and a permutation whose descent set
  is a subset of $\dent(\invtab) \setminus \{ 0 \}$.
\end{proposition}

\begin{proof}
  This proof is similar to that of Proposition~\ref{bk-equiv-class-size}.
  Define $S = \{ s_1, s_2, \dots, s_k\}$ to be the set of distinct
  entries in $\invtab$ with the exception of $0$.
  $$S = \dent(\invtab) \setminus \{ 0 \}.
  $$
  From the definition of the decomposition, matrices in
  $[\invtab]_\alequiv$ satisfy that columns $n-s_1$, $n-s_2$, \dots,
  $n-s_k$ are non-empty.

  Recall that, for a ballot matrix fixed under $\eta$, if there is an
  ascent at position $i$, $a_i < a_{i+1}$, then column $i$ must be
  non-empty. If there is a descent, then it may or may not be
  non-empty. Therefore the set of ascent positions in the associated
  permutation must be a subset of $n-s_1, n-s_2, \dots, n-s_k$.
  Trivially, reversing such a permutation yields a permutation whose
  descent set is a subset of ${s_1, s_2, \dots, s_k}$.

  Therefore, for any given inversion table where the distinct entries is
  $\{0\} \cup S$, the set of permutations which can be associated
  are trivially in bijection with those where the descent set is a
  subset of $S$.
\end{proof}

\begin{theorem}
  Labeled interval orders on $[n]$ are in bijection with the set
  $$\sum_{S\subseteq [n-1]} \balpha_n(S) \times \blambda_n(S).$$
  This set may be alternatively written as
  $$
  \{ (\pi, \tau) \in \sym_n \times \sym_n : D(\pi) \subseteq
  A(\tau) \}.
  $$
\end{theorem}

\begin{proof}
  Corollary~\ref{lambda-inversion-table-alternative} gives that
  permutations in $\blambda_n(S)$ are in bijection with inversion
  tables with set of distinct elements $\{0\} \cup
  S$. Proposition~\ref{al-equiv-class-size} states that the
  permutations with which an inversion table $\invtab$ can be paired
  are those with their descent set a subset of $\dent(\invtab)
  \setminus \{0\}$. From definition, such permutations are those
  contained within $\balpha_n(S)$.
\end{proof}


\bibliographystyle{plain}

\end{document}